\documentclass[letterpaper, 10 pt, conference]{ieeeconf}

\newtheorem{theorem}{Theorem}
\newtheorem{example}[theorem]{Example}
\newtheorem{lemma}[theorem]{Lemma}
\newtheorem{remark}[theorem]{Remark}
\let\newexample=\example
\renewenvironment{example}{\newexample\sl}{\endtheorem}
\usepackage{latin1}
\usepackage{eufrak}
\usepackage[mathscr]{euscript}
\usepackage{graphics} 
\usepackage{epsfig} 

\newenvironment{block}{\left[ \begin{array}}{\end{array}\right] }
\def\ip#1#2{\left\langle #1, #2 \right\rangle }

\font\BBbannan=msbm10 at 10pt

\def\triangleeq{\stackrel{\triangle }{=}}

\def\Reella{\mbox{\BBbannan R} } 
\def\Complexa{\mbox{\BBbannan C} }

\def\trace{\mbox{tr} }

\def\cA{{\cal A } }
\def\cB{{\cal B } }
\def\cC{{\cal C } }
\def\cD{{\cal D } }

\def\cG{{\cal G } }
\def\cH{{\cal H } }
\def\cK{{\cal K } }
\def\cL{{\cal L } }
\def\cM{{\cal M } }

\def\cS{{\cal S } }

\def\cW{{\cal W } }

\def\ninputs{1}
\def\noutputs{1}
\def\nstates{\eta}

\title{\LARGE \bf
Generalizing the Markov and covariance 
interpolation problem using input-to-state filters
}

\author{Per Enqvist
\thanks{This work was supported by the National Research Council ``Vetenskapsrådet'' in Sweden}
\thanks{P. Enqvist is with the department of Mathematics,
        Royal Institute of Technology, Lindstedtsv. 25, SE-100 44 Stockholm, Sweden.
        {\tt\small penqvist@math.kth.se}}%
}

\begin{document}

\maketitle
\thispagestyle{empty}
\pagestyle{empty}

\begin{abstract}
In the Markov and covariance interpolation problem
a transfer function $W$ is sought that match 
the first coefficients in the expansion   
of $W$ around zero and the first coefficients
of the Laurent expansion of the corresponding 
spectral density $WW^\star$.
Here we solve an interpolation
problem where the matched parameters are the 
coefficients of expansions of $W$ and $WW^\star$ 
around various points in the disc.
The solution is derived using input-to-state filters and 
is determined by simple calculations
such as solving Lyapunov equations and generalized
eigenvalue problems.
\end{abstract}


\section{INTRODUCTION}

The problem of designing filters from covariances and 
Markov parameters has been studied before in numerous 
papers \cite{MulR,Inouye,Skelton88,Skelton93,Liu_Skelton,Skelton_And88,Skelton96,MK}.
Skelton {\emph et. al.} call a stable model matching Markov parameters
$H_0,H_1,\dots,H_{q-1}$ and covariances $R_0,R_1,\dots,R_{q-1}$
a q-Markov COVariance Equivalent Realization (q-Markov COVER)
 and they have shown that if the data satisfies a particular 
consistency condition (which can be avoided using a 
variable input variance as in \cite{cdc2002}), there are many 
such q-Markov COVERs and they are parameterized by 
a set of unitary matrices.
One of the parameters considered ``as known'' in the classical 
q-Markov COVER theory is the variance of the input noise.
In \cite{cdc2002,PerIEEE-AC} the author proposed a method 
for designing minimal degree realizations using 
the variance of the input noise as a design parameter
which enabled a realization of lower degree to be determined.
In fact, that method guarantees a unique stable solution for 
generic data.
Here, using  input-to-state filters, we solve an interpolation
problem where the matched parameters are the 
coefficients of expansions of $W$ and $WW^\star$ 
around various points in the disc.
We could for example consider matching the constraints
\[ W(p_1) =q_1, \cdots , W(p_n)=q_n\]
for some points $p_1,\cdots,p_n$ in the unit disc
and similarly for $WW^\star$.
Most results in \cite{cdc2002} carry over to this more 
general problem.
A formal definition of the problem considered is given 
in the next section.
Another approach to this problem was taken in \cite{mboup}.
The main objective of that paper was to prove existence 
of a fixed point for the Steiglitz-McBride algorithm and 
a different kind of normalization was used.


\section{Problem formulation}\label{sec:pro}

First the Markov and Covariance interpolation problem 
as formulated in \cite{cdc2002} is described and then 
input-to-state filters are introduced for treating the 
generalized problem.

\subsection{The Markov and Covariance interpolation problem}
\label{sec:MC}


We consider a SISO system where a deterministic control signal $v$ 
and a stochastic noise signal $w$ are fed through the 
same system $\cW$ to produce the output $y$ as depicted in 
Fig.~\ref{f1}.

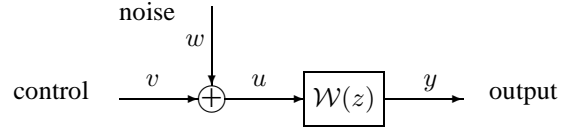
\begin{figure}
\begin{picture}(200,50)(-10,-10)
\put(40,30){\mbox{noise}}
\put(65,20){$w$}
\put(0,0){\mbox{control}}
\put(50,5){$v$}
\put(40,0){\vector(1,0){30}}
\put(75,0){\circle{10}}
\put(71,-3){\Large +}
\put(80,0){\vector(1,0){30}}
\put(90,5){$u$}
\put(75,35){\vector(0,-1){30}}
\put(110,-10){\framebox(30,20){$\cW(z)$}}
\put(140,0){\vector(1,0){30}}
\put(180,0){\mbox{output}}
\put(155,5){$y$}
\end{picture}
\caption{System considered}  \label{f1}
\end{figure}

Define $u=v+w$ and let $v$ be the control input 
and $w$ an additive noise term. 
Assuming that the transfer function $\cW$ is rational and 
of McMillan degree $\nstates$, 
it can be described by a minimal state space system
\begin{equation}\label{eq:ss}
\begin{array}{rcl}
\chi_{j+1} &  = & \cA \chi_j +  \cB u_j, \\ 
y_j   & = & \cC \chi_j +  \cD u_j,
\end{array}
\end{equation}
 where $\cA\in\Complexa^{\nstates \times \nstates}$, 
$\cB\in\Complexa^{\nstates \times \ninputs}$, 
$\cC\in\Complexa^{\noutputs \times \nstates}$ and 
$\cD\in \Complexa$.
The output $y$ is the superposition of the outputs
due to each of the inputs $v$ and $w$.
Therefore, data from the system 
can be obtained by the following idealized 
experiments - or in any other practically more 
suitable way.

First, determine the 
output when the noise $w$ is zero and $v$ is a unit impulse, 
yielding the Markov parameters (impulse response parameters)
\begin{equation}\label{mar} 
H_0, H_1, \dots, H_\ell. 
\end{equation} 

Second, determine the output when the control $v$ is 
zero and $w$ is mean zero white noise with unknown
variance $\Lambda \in \Complexa$.
Assuming $\cW$ is asymptotically stable then
this system provides a realization of a stationary 
stochastic process, and by truncated ergodic sums 
the covariances 
\begin{equation}\label{cov}
 R_0, R_1, \dots, R_\ell 
\end{equation}
can be estimated such that the condition
\begin{equation}\label{T}
\left[ \begin{array}{cccc}
R_0 & R_1 & \ddots & R_\ell \\
\bar R_1 & R_0 & \ddots & \ddots \\
\ddots & \ddots & \ddots & R_1\\
\bar R_\ell & \ddots & \bar R_1 & R_0
\end{array} \right] \succeq 0
\end{equation}
is satisfied.


\subsection{Input-to-State filters and interpolation}

In order to analyze a signal it is useful 
to consider a new signal obtained by 
applying an input-to-state filter \cite{TGselHA},
i.e. if $y_k$ is our original signal 
we define the new state vector $x_k$ by
\begin{equation}\label{i2s}
 x_k = A x_{k-1} + By_k,  \quad x_0=0,
\end{equation}
where $A\in\Complexa^{n\times n}$, $B\in\Complexa^{n\times 1}$
and the eigenvalues of $A$ lies in the open unit disc.
Note that the state $\chi$ in (\ref{eq:ss})
is not the same as the new ``artificial'' state $x$
defined in (\ref{i2s}) from the ``inputs'' $y_k$.

Consider the input to state map $G$:

 \begin{equation}
 G(z) := (I-zA)^{-1} B
 \end{equation}
 where we will assume that $(A,B)$ is a reachable pair,
i.e. 
\begin{equation}\label{Gamma}
 \Gamma = \begin{block}{cccc} 
B& AB & \dots & A^{n-1}B
\end{block}
\end{equation}
is full rank. 

A wide class of interpolation
problems can now be approached in a unified framework
by expressing the interpolation constraints as inner
products with the input-to-state map $G$.
Let $\ip{\cdot}{\cdot}$ denote the standard $L_2$ inner
 product on the circle, 
and for vector- and matrix-valued functions $F_1$ and $F_2$
define
\[ \ip{F_1}{F_2} = \frac{1}{2\pi} \int_{-\pi}^\pi 
F_1(e^{i\theta}) F_2^\star(e^{i\theta}) \, d\theta, \]
where the integral is evaluated elementwise and 
$F^\star$ denotes the adjoint of $G$, i.e.
\[ F^\star(z) = F(\bar z^{-1})^*, \]
where the superscript $*$ denotes the usual complex conjugate.
Notice that we will allow the $L_2$ inner product 
between two matrix-valued functions $F_1$ and $F_2$, 
possibly of different sizes,
provided that the product $F_1(z)F_2^\star(z)$ is well defined. 

In the special case where $A=\mbox{diag} (p_1,p_2,\dots,p_n)$, $|p_j|<1$, 
and $B=[1,1,\cdots,1]^T$, 
the scalar function on the $k$:th row of $G(z)$ is
\[  g_k(z) =\frac{1}{1-p_i z}.\]
Then from the Cauchy's integral formula  
$\ip{f}{g_k}=f(p_i^*)$, i.e. the values of $f$ 
at the selected points can be expressed in terms of 
the inner product.

In the other special case where all interpolation points are at 
the origin, i.e. the values of the function and its derivatives
at zero are interpolated as in the Caratheodory
interpolation problem, then we could chose 
\begin{equation}\label{AB_CM}
 A= \begin{block}{cccc}
  0 & 0 & \dots & 0\\
  1 & 0 & \dots & 0\\
  0 & \ddots & \ddots & \vdots\\
  0 & 0 & 1 & 0  
  \end{block}, \quad 
 B= \begin{block}{cccc} 1 \\ 0 \\ \vdots \\ 0 \end{block}, 
\end{equation}
so that  
\begin{equation}
 G(z) = \begin{block}{cccc} 1 \\ z \\ \vdots \\ z^{n-1} 
\end{block}. 
\end{equation}
The states are then the $n$ most recent outputs and it is easy 
to see that the covariance of the state is a Toeplitz matrix
as in (\ref{T}).

In practice one could be interested in having 
a mixture of interpolation conditions on the function
values at different points and on some of its derivatives, 
and this can be accomplished by considering for example
$A$-matrices with some particular Jordan structure.
To be able to find a $B$ such that $(A,B)$ is reachable
it is necessary that $A$ is cyclic, so there can not 
be more than one Jordan block for each interpolation
point (eigenvalue of $A$).

Now given some $G$, if $d\mu$ is a matricial spectral measure 
of the input (i.e. $y$ the input to $G$)
 the state covariance $\Sigma$ will satisfy \cite{TGselHA}

 \begin{equation}\label{Sigma}
  \Sigma = \int_{-\pi}^\pi G(e^{j\theta}) d\mu(\theta) G(e^{j\theta})^\star.
 \end{equation} 

For the more general input-to-state filter it is more difficult 
to know what is the structure of the state-covariance matrix.
In Theorem~\ref{th:Sigma} below,  a  
result from \cite{TGselHA} describing the 
feasible structures is stated, but first we need to remind
the reader of a well-known result.

\begin{lemma}\label{lem:G}
The matrix $\cG$ defined by 
\begin{equation}\label{GG}
 \cG \triangleeq \ip{G}{G}, 
\end{equation}
is the Reachability Gramian solving the 
discrete time Lyapunov equation
\begin{equation} \label{AGA}
 \cG = A \cG A^* + BB^*. 
\end{equation}
Since $(A,B)$ is assumed to be a reachable pair, 
$\cG$ is invertible.
\end{lemma}

\begin{proof}
Note first that 
\begin{equation}\label{AG}
 AG(z) = z^{-1}( G(z) -  B),
\end{equation}
and then multiply (\ref{GG})
with $A$ from the left and $A^*$ from the right 
to obtain
\begin{eqnarray}
A\cG A^*  & = & \ip{AGG^\star A^*}{1} \nonumber \\
& = & 
\ip{z^{-1} (G -B) z(G^\star-B^*)}{1} 
\nonumber \\
& = & 
\ip{GG^\star +BB^*-BG^\star-GB^*}{1}
\nonumber \\
& = & 
\cG+ BB^*-BB^*-BB^*.
\label{eq:AGA}
\end{eqnarray}
The last step follows by observing that 
 $G$ is analytic in the unit disc and thus
$\ip{G}{1}=G(0) = B$, and similarly
$\ip{G^\star}{1}=G^\star(\infty) = B^*$.

Since $A$ is assumed to be asymptotically stable
the solution to the Lyapunov equation is unique, 
and this completes the proof.
\end{proof}

\begin{theorem}\label{th:Sigma}
A positive definite matrix $\Sigma$ is a state-covariance
matrix for a suitable input process if and only if it is of the 
form
\[ \Sigma = \frac{1}{2} (\cM \cG + \cG \cM^*)\]
for a matrix $\cM$ which commutes with $A$.
Furthermore, any such matrix $\cM$ is uniquely defined modulo
an additive imaginary constant $\alpha I$ with $\alpha\in j\Reella$.
\end{theorem}

Another way to describe the structure of the state covariance $\Sigma$ 
is that it satisfies the equation \cite{TGstatecov}
\[ \Sigma - A\Sigma A^* = BL+L^*B^* \]
for some $L$.


Let $\cH_2$ denote the Hardy space of functions that are analytic 
in the unit disc with square-integrable radial limits, 
and define
\begin{equation}
\cK \triangleeq \cH_2 \ominus b(z) \cH_2,
\end{equation}
where 
$b(z) = \det(zI-A^*) / \det(I-zA)$
is a Blaschke product with poles at the eigenvalues of $A$.
In fact, $b(z)$ is the inner, or Douglas-Shapiro-Shields, 
factor of $G(z)$.
Then $\cK$ contains all functions in $\cH_2$ which 
are orthogonal to those that 
vanish on the spectrum of $A^*$, and it is 
usually called the coinvariant subspace.
By \cite[Prop. 4]{TGstructure} the elements of $G(z)$
form a basis for $\cK$, so any $f\in\cK$ 
can be written $f(z)=CG(z)$ for some vector $C$, and then 
\begin{equation}\label{ss2tf}
 f(z) = \frac{\det(I-z(A-BC))-\det(I-zA)}{\det(I-zA)} \in \cK.
\end{equation}

We also need to take inner products between elements in $\cK$ 
and $\cH_2$, and then the following formulas are useful.

\begin{lemma}
If $f(z)\in \cH_2$ then 
\begin{eqnarray}\label{Geval1}
\ip{f}{G}
&= &B^*f(A^*)
\end{eqnarray} 
and
\begin{eqnarray}\label{Geval2}
\ip{G}{f} 
&= & \bar f(A) B.
\end{eqnarray} 
Furthermore, 
\begin{eqnarray}\label{Geval3}
\ip{G}{fG} 
&= & \bar f(A) \cG.
\end{eqnarray} 
\end{lemma}
Note that it is important here that $f$ is a scalar function.
\begin{proof}
Since $f\in\cH_2$  and $G\in\cH_2^{n\times 1}$ they have series expansions
\[ f(z) = \sum_{k=0}^\infty f_k z^k, \]
and
\[ G(z) = \sum_{k=0}^\infty G_k z^k
        = \sum_{k=0}^\infty A^k B z^k. \]
Then 
\[ \ip{G}{f} = \sum_{k=0}^\infty \sum_{\ell=0}^\infty
\ip{A^kBz^k}{f_\ell z^\ell} 
= \sum_{k=0}^\infty \bar{f}_k A^k B = \bar{f}(A)B, \]
and the formula for $\ip{f}{G}$ follows by considering 
the complex conjugate.

Finally, 
\begin{eqnarray*}
\ip{G}{fG} & =&   \sum_{k=0}^\infty \sum_{\ell=0}^\infty
\sum_{m=0}^\infty
\ip{A^kBz^k}{f_\ell z^\ell G_m z^m }\\
 & =&   \sum_{\ell=0}^\infty \sum_{m=0}^\infty
 A^{\ell+m} B  \bar f_\ell G_m^*  \\
& =&   \sum_{\ell=0}^\infty \bar f_\ell A^\ell
\sum_{m=0}^\infty
 A^{m} B B^* (A^m)^* \\
& = & \bar f(A) \cG, 
\end{eqnarray*}
which concludes the proof.
\end{proof}

\begin{remark}
Alternatively, this could be proven by considering 
generalized Cauchy kernels 
\[ \ip{G}{f} = \left(
\frac{1}{2\pi} \int_{-\pi}^\pi \bar f(e^{-j\theta})
(I-e^{j\theta}A)^{-1} \, d\theta
\right) B, \]
as in \cite{TGselHA}.
\end{remark}

Estimation of the parameters from data can be performed
by applying the input-to-state filter and 
then using standard techniques, see \cite{BGL}
for examples of filter bank data analysis.





\section{The global optimization problem}

We will assume here that the spectral measure in (\ref{Sigma})
is given by
\[ d\mu (\theta) = W(e^{i\theta}) \:\Lambda d\theta\: 
W(e^{i\theta})^\star, \]
where 
\begin{equation}
W(z) = \sum_{k=0}^\infty w_k z^k \in \cH_2,
\end{equation}
i.e. is analytic in the unit disc 
(so the sum converges for all $z$ in the unit disc), 
and in this class of spectral measures we will 
find the one allowing the maximal 
input variance $\Lambda$ meanwhile 
satisfying the following interpolation conditions:
\begin{equation}\label{covconstr}
\ip{GW\Lambda W^\star G^\star}{1}= \Sigma,  
\end{equation}
where the state covariance $\Sigma$ satisfies the condition in 
Theorem~\ref{th:Sigma}, and
\begin{equation} \label{markconstr}
\ip{G}{W}=H, 
\end{equation}
for an arbitrary nonzero state-Markov vector $H$.

The interpolation constraint in (\ref{covconstr})
was considered in, for example,
\cite{TGstatecov}.
The interpolation constraint in (\ref{markconstr})
can be recognized as a special 
case of the Lagrange-Sylvester interpolation
as studied in 
\cite[section 16]{BGR}.
Here, both constraints are enforced simultaneously.

Thus the optimization problem considered is:
\[ (\mathfrak S) \quad  \begin{block}{rl}
{\displaystyle  \max_{ 
 \begin{array}{c} W\in\cH_2\\
                        \Lambda \in \Reella^+\end{array} } }&
\Lambda, \\[2ex]
\mbox{s.t.} & \left\{
\begin{array}{lll}
\ip{GW\Lambda W^\star G^\star}{1}= \Sigma,  \\
\ip{G}{W}=H. 
\end{array}\right. 
\end{block} \]

Let $\Xi$ be an $(n \times n)$ Hermitian  matrix 
and $\zeta$ be an $(1\times n)$ vector consisting of 
Lagrange multipliers, 
the Lagrangian is then
\begin{eqnarray*}
\cL(W,\Lambda) & \hspace{-0.4cm} \triangleeq  \hspace{-0.4cm} & 
\Lambda  +  
\trace\left\{ 
(\Sigma-\ip{GW\Lambda W^\star G^\star}{1}) \Xi
\right\} 
\\
&& + \zeta (\ip{G}{W}- H).
\end{eqnarray*}

We can rewrite it in the following form
\begin{eqnarray*}
\cL(W,\Lambda) &  \hspace{-0.4cm} =  \hspace{-0.4cm} 
& \Lambda  +  
\trace\left\{ \Sigma \Xi\right\} 
-\ip{W\Lambda W^\star }{G^\star\Xi G}\\
&& -
\zeta H
+\ip{\zeta G}{W}.
\end{eqnarray*}
where
\[ G^\star\Xi G =
B^*(I-\bar z^{-1}A^*)^{-1}\Xi (I- z A)^{-1}B \]
and
\[ \zeta G = \zeta (I- zA)^{-1}B\]
are scalar functions.

Before taking the maximum we write it in the form
\begin{eqnarray*}
\cL(W,\Lambda) &  =  & 
\ip{\zeta G}{W}+\ip{(1-W^\star G^\star\Xi GW)\Lambda}{1}
\\
&& + \trace\{\Sigma \Xi\} - \zeta H
\end{eqnarray*}

\medskip\noindent
{\bf Note:}
$\mbox{\rm Sup} \{ \cL(W,\Lambda) | W\in\cH_2,
                                     \Lambda > 0\} < \infty$ 
only if $G^\star\Xi G$ is in the ``positive cone'', i.e.
it is non-negative for all $z$ on the unit circle,
and 
\begin{equation}\label{WQW2}
\ip{W^\star G^\star \Xi GW}{1} \geq 1.
\end{equation}

Maximizing over $\Lambda$ while assuming (\ref{WQW2}) 
it must hold that
\begin{equation}
\Lambda \ip{1-W^\star G^\star\Xi GW}{1} 
=0,
\end{equation}
and since $\Lambda \neq 0$,
equality must hold in (\ref{WQW2}), i.e.
\begin{equation}\label{WQWstrict2}
\ip{W^\star G^\star\Xi GW}{1} = 1.
\end{equation}

Maximizing over $W$ shows that the following variation has to be 
zero for all $\delta W$ 
\begin{equation}
 \ip{\zeta G-2\Lambda G^\star\Xi GW}{\delta W }  =0.
\end{equation}
Therefore 
\begin{equation}\label{P2} 
\zeta G =2\Lambda  G^\star\Xi G W + V^\star 
\end{equation} 
where $V\in \cH_2$ and $V(0)=0$.
From (\ref{P2}) the poles of $V^\star$ has to be poles 
of $G^\star$. Furthermore, $V\in\cK$ follows by considering the 
partial fraction expansion of (\ref{P2}), so
there must be a 
vector $\nu$ such that
\begin{equation} \label{V} V(z) = \nu G(z) \end{equation}
Then the transfer function $W$ will be given by 
\[ W = \frac{1}{2} (G^\star\Xi G)^{-1} (\zeta G  - V^\star) \Lambda^{-1} \]
\[ \quad =\frac{1}{2\Lambda} (G^\star\Xi G)^{-1} (\zeta G  - G^\star\nu^*) .\]

\begin{lemma}
If $\Xi$ is non-negative 
we can factor $G^\star\Xi G$ as 
\begin{equation}\label{QAA2}
 G^\star \Xi G =  (\xi G)^\star (\xi G),
\end{equation}
where $\xi$ is a row-vector.
\end{lemma}

\begin{proof}
Since $\Xi$ is non-negative and Hermitian it can be factorized as
\[ \Xi = 
\begin{block}{cccc}
\xi_1^* & \xi_2^* & \cdots & \xi_n^*
\end{block}
\begin{block}{cccc}
\xi_1 \\ \xi_2 \\ \vdots \\ \xi_n
\end{block}.
\]
Then 
\begin{equation}\label{GxiG}
 G^\star \Xi G = \sum_{k=1}^n G^\star \xi^*_k \xi_k G 
\end{equation}
and $G^\star \Xi G$ is a sum of elements in $\cK \cup \cK^\star$,
where $k^\star\in\cK^\star$ if $k \in \cK$.
Since all the terms in (\ref{GxiG})
are positive, by spectral factorization 
a vector $\xi$  such that the sum  
is equal to $G^\star \xi^* \xi G$ can be found.
\end{proof}

For $W$ to be analytic outside the unit disc it is necessary that
the factor $(\xi G)^{-\star}$ is cancelled, i.e.
 we need that
\[ W = \frac{1}{2\Lambda}
\frac{\zeta G  - V^\star}{ (\xi G)(\xi G)^{\star}}
     =\frac{\sigma G}{\xi G}.\]
From (\ref{P2}) and (\ref{QAA2}) it follows that
\begin{equation}\label{eq:etaG}
\zeta G =2\Lambda G^\star \xi^* \sigma G+ G^\star\nu^*.
\end{equation}

Using (\ref{WQWstrict2}) and (\ref{eq:etaG})
the dual function is 

\begin{eqnarray*}
\varphi (\xi,\sigma) & \hspace{-0.2cm} =  \hspace{-0.2cm} & 
 \ip{\zeta G }{\frac{\sigma G}{\xi G}}
+\trace\{ \Sigma \xi^* \xi \} - \zeta H
\\[2ex]
&\hspace{-0.2cm} = \hspace{-0.2cm} &
 2\Lambda \sigma \cG \sigma^*
+\xi \Sigma \xi^* - \zeta H
\end{eqnarray*}
since $\ip{V^\star}{W}=0$ and
where $\cG$ was defined in (\ref{GG}).

To determine the last term $\zeta H$, multiply 
(\ref{eq:etaG}) with $G^\star$ and integrate to obtain:
\[
\zeta \ip{GG^\star}{1} =2\Lambda  \ip{G^\star\xi^*\sigma G G^\star}{1}  +
\ip{V^\star G^\star}{1}
\]
the last term is zero and then
\[ \zeta H =  2\Lambda  \sigma \ip{G(G^\star\cG^{-1}H)G^\star}{1}\xi^*
= 2\Lambda \sigma \cH \xi^*,
\]
where
\begin{equation} \label{eq:cH}
 \cH \triangleeq  \ip{G(G^\star\cG^{-1}H)G^\star}{1} .
\end{equation}

\begin{lemma}
The matrix $\cH$ defined by (\ref{eq:cH})
is the unique solution to the Stein equation 
\begin{equation}\label{AHA}
\cH  = A \cH A^* + HB^*.
\end{equation}
\end{lemma}
\begin{proof}
As in the proof of Lemma~\ref{lem:G}, note that
(\ref{AG}) holds 
and then multiply (\ref{eq:cH})
with $A$ from the left and $A^*$ from the right 
to obtain
\begin{eqnarray}
A\cH A^*  \hspace{-2mm} & = &  
\hspace{-2mm}\ip{AGG^\star A^*(G^\star\cG^{-1}H)}{1} \nonumber \\
& = & 
\hspace{-2mm}\ip{z^{-1} (G - B)z(G^\star-B^*)(G^\star\cG^{-1}H)}{1} 
\nonumber \\
& = & 
\hspace{-2mm}\ip{(GG^\star +BB^*-BG^\star-GB^*)(G^*\cG^{-1}H)}{1}
\nonumber \\
& = & 
\hspace{-2mm}\cH+ B\ip{(B^*-G^\star)(G^\star\cG^{-1}H)}{1}
\nonumber \\
&  & -\ip{GB^*(G^\star\cG^{-1}H)}{1}   \label{eq:AHA}
\end{eqnarray}
The second term in (\ref{eq:AHA}) is zero since
the integrand is analytic outside the unit circle
and $G^\star(\infty)=B^*$.

The third term in (\ref{eq:AHA})
is $HB^*$, which follows by 
considering the action on an arbitrary vector $v$;
\begin{eqnarray*}
\ip{GB^*(G^\star\cG^{-1}H)}{1}v  & = & 
\ip{GB^*v(G^\star\cG^{-1}H)}{1}  \\
& = &
\ip{GG^\star\cG^{-1}H}{1}B^*v  \\
& = & HB^* v.
\end{eqnarray*}
Since $A$ is assumed to be asymptotically stable
the solution to the Stein equation is unique, 
and this completes the proof.
\end{proof}



\begin{remark}
Note that if 
 \begin{equation}
 h(z) := (I-zA)^{-1} H
 \end{equation}
then $\ip{h}{G}$
solves the Stein equation (\ref{AHA}),
and then by uniqueness (compare \cite[Eq. (40)]{TGstructure})
\[ \cH = \ip{h}{G}. \]
\end{remark}

Now, the dual optimality function is
\begin{eqnarray} 
\varphi (\xi,\sigma) \hspace{-0.2cm} & = &\hspace{-0.2cm}
 2\Lambda \sigma \cG \sigma^*
+\xi \Sigma \xi^* - 2\Lambda \sigma \cH \xi^*
\nonumber \\[1ex]
& = &\hspace{-0.2cm}
\begin{block}{cc}
\sigma & \xi 
\end{block}
\begin{block}{cc}
2\Lambda\cG & -2\Lambda \cH\\
0 & \Sigma 
\end{block}
\begin{block}{cc}
\sigma^* \\ \xi^* 
\end{block}
\end{eqnarray}

Maximizing this expression over positive $\Lambda$ 
\[ \varphi (\xi,\sigma) =
 2\Lambda (\sigma \cG \sigma^* - \sigma \cH \xi^*)
+\xi \Sigma \xi^* \]
it is clear that 
$(\sigma \cG \sigma^* - \sigma \cH \xi^*)$ 
has to be zero.
In fact, if it is negative the optimal value of $\Lambda$ would
be zero and we have assumed that it is positive, and 
if it is positive the optimal value of $\Lambda$ would
be infinite. 
Furthermore, the following holds:
\begin{lemma}
Given that $W(z) = \frac{\sigma G(z)}{\xi G(z)}\in \cH_2$, 
the constraint $\ip{G}{W}=H$ implies that 
$\cH \xi^* = \cG \sigma^*$.
\end{lemma}

\begin{proof}
We know that $\ip{G}{W} = H$, and thus

\[ G^\star\cG^{-1}\ip{G}{\frac{\sigma G}{\xi G}} = G^\star\cG^{-1}H,\]
is a scalar function, so
\begin{eqnarray*} 
\cH \xi^* & = & 
\ip{GG^\star (G^\star\cG^{-1}H)}{1}\xi^* \\
& = & 
\ip{ G G^\star\xi^* \left( G^\star\cG^{-1}\ip{G}{\frac{\sigma G}{\xi G}}
\right)}{1} \\
& = & 
\ip{ GG^\star}{\xi G}\cG^{-1}\ip{G}{\frac{\sigma G}{\xi G}} \\
& = & 
\ip{ G}{(\xi G)G}\cG^{-1}\ip{G}{\frac{1}{\xi G}G}\sigma^* \\
& = & \overline{(\xi G(A))} \cG \cG^{-1} 
     (\overline{\xi G(A)})^{-1} \cG \sigma^*\\
& = & \cG \sigma^*
\end{eqnarray*}
where we have used (\ref{Geval3}) twice.
\end{proof}


The complementarity condition
(\ref{WQWstrict2}) can be formulated as 
\[ \ip{\frac{G^\star \sigma^*}{G^\star\xi^*}G^\star\Xi G
\frac{\sigma G}{\xi G}}{1}   =\sigma \cG \sigma^* = 1,\]
and then the dual problem is 
\[ (\mathfrak D) \quad \begin{block}{rl}
{\displaystyle  \min_{\sigma,\xi}}&
\xi \Sigma \xi^*
\\[2ex]
\mbox{s.t.} & \left\{
\begin{array}{lll}
\cG \sigma^* =  \cH \xi^*, \\
\sigma \cG \sigma^* = 1
\end{array}\right. 
\end{block} \]
where $\sigma$ and $\xi$ are related by 
the Markov interpolation conditions.

The variable $\Lambda$ was eliminated above, but it can recovered by
considering the dual of the dual.
Let $\Lambda$ be the 
Lagrange multiplier and
use $\cG \sigma^* = \cH \xi^*$ to eliminate $\sigma$
\begin{eqnarray*} 
L  & = & \xi \Sigma \xi^*
     - \Lambda(\sigma\cG\sigma^*-1) \\
& = &  \xi \Sigma \xi^*
-\Lambda\xi\cH^*\cG^{-1}\cH\xi^*
+\Lambda
\\
& = & 
\xi \left( \Sigma -\Lambda\cH^*\cG^{-1}\cH\right)  \xi^*
+\Lambda
\end{eqnarray*}

which leads us to maximize  $\Lambda$ as 
$\Sigma -\Lambda\cH\cG^{-1}\cH^*$
is non-negative, i.e.

\[ (\mathfrak P^*) \quad \begin{block}{rl}
{\displaystyle  \max_{\Lambda}}&
\Lambda, \\[2ex]
\mbox{s.t.} &
\Sigma -\Lambda\cH^*\cG^{-1}\cH \succeq 0
\end{block} \]
The optimal $\Lambda$ is now given by the largest positive value 
such that  $\Sigma - \Lambda \cH^*\cG^{-1} \cH$ is non-negative 
definite, i.e. the smallest generalized eigenvalue of 
 $(\Sigma,\cH^*\cG^{-1} \cH)$.

\begin{theorem}
Given a state covariance $\Sigma$ satisfying  the condition in 
Theorem~\ref{th:Sigma}, and an arbitrary nonzero state-Markov vector $H$.
Then, an optimizer  $W$ to problem 
$(\mathfrak S)$ is given by $W(z)=(\sigma G(z))/(\xi G(z))$,
where $\xi$ is a nonzero solution to the equation 
\[ \left( \Sigma - \Lambda \cH^*\cG^{-1} \cH \right) \xi^* =0, \]
$\Lambda$ is the smallest generalized eigenvalue of 
 $(\Sigma,\cH^*\cG^{-1} \cH)$,
and finally $\sigma$ is determined by
\[ \sigma^* = \cG^{-1} \cH \xi^* .\]
Furthermore, if the smallest generalized eigenvalue has multiplicity
one the optimizer $W$ is unique.
\end{theorem}

\begin{proof}
This follows from the derivation above
\end{proof}

As in the Markov and Covariance interpolation
problem there is a special choice of the 
Markov parameters that reduce the problem to 
the equivalent of the Pisarenko method \cite{TG-P}.
Namely, chosing $H=B$, then $\cH=\cG$ and maximizing $\Lambda$ under
the constraint
\[ \Sigma-\Lambda\cH^*\cG^{-1}\cH = \Sigma - \Lambda \cG \succeq 0\]
makes the corresponding $W$ an inner function \cite{TG-P}.

\section{Model reduction example}

To illustrate the method proposed here, a model reduction 
application is considered.
The method proposed here is a generalization 
of the q-Markov COVER methods, that were initially
proposed to be used for model reduction \cite{Skelton88}.

The transfer function from input 2 to output 1 of 
a portable CD-player is considered.
This model, of order 120, is provided
by SLICOT \cite{CV}, and has been used 
by, for example, \cite{GA,Vanna}.
The magnitude of the transfer function 
is depicted with a thick solid green line in 
Figure~\ref{frq500}.
There is a wide range of frequencies over 
which there are interesting features of the 
Bode plot.

\begin{figure*}[!thb]
      \centerline{  \resizebox{12cm}{!}{%
      \includegraphics{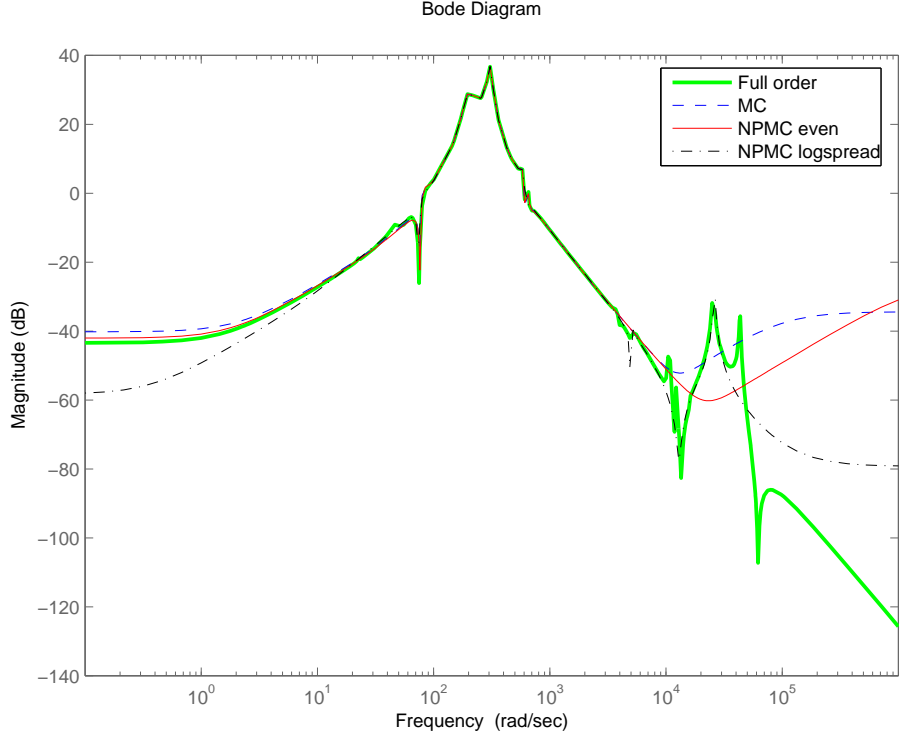}}}
      \caption{Model reduction} 
      \label{frq500}
\end{figure*}

The given transfer function is a  continuous time stable function.
The bilinear map 
\begin{equation}\label{bilin}
 z = \frac{1-sT/2}{1+sT/2}, 
\end{equation}
where $T=1/250$, is used to transform 
the continuous time model into a discrete
time model.

Using three different input-to-state filters, 
reduced order models of degree $13$ 
are designed.
Our aim will not be to find the optimal 
interpolation point locations for this particular model,
 but to illustrate the way this choice effects the 
solutions.

First an input-to-state filter as in (\ref{AB_CM})
was applied, corresponding to the Markov and covariance
interpolation problem described in section \ref{sec:MC}, 
 and the  magnitude plot of the resulting model
is depicted with a blue dashed line 
in Figure~\ref{frq500}.

Then, an input-to-state filter with 14 poles spread 
evenly around a circle with radius $0.95$ was applied.
The magnitude plot, depicted in Figure~\ref{frq500}
with a solid red line,
is similar to the first one, but with a slightly
smaller error for low and high frequencies.

Finally, an  input-to-state filter with 14 poles spread 
unevenly around a circle with radius $0.9$ was applied.
The spread in frequencies were chosen to correpond to 
a logarithmic spread in the frequency interval $10^1$ to $10^5$.
In the discrete domain, the interpolation point 
locations are depicted with black plusses in 
Figure~\ref{frq500ip} together with the 
interpolation points of the two other filters.
This choice of interpolation points is made to 
compensate for the frequency warping caused by the 
bilinear map (\ref{bilin}).
The magnitude plot, depicted in Figure~\ref{frq500}
with a black dashed-dotted line,
shows an improvement of the fit in the frequency
range where the poles were chosen.

\begin{figure}[!htb]
      \centerline{  \resizebox{8cm}{!}{%
      \includegraphics{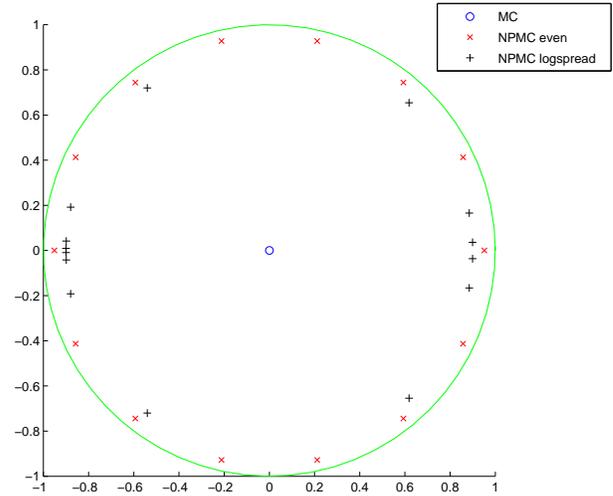}}}
      \caption{Interpolation point locations} 
      \label{frq500ip}
\end{figure}

For comparison, a model of degree 13
is determined using a  standard balanced truncation 
model reduction method and its
magnitude plot is depicted in Figure \ref{bal}.
A good fit for the interval of frequencies 
where the magnitude is large is obtained. 
It is well known that weights can be applied to improve the 
fit for certain frequency regions.
The choice of these weights, as well as the choice of
interpolation points in our approach, should 
be made with the prior knowledge 
and requirements of the low order 
model in mind.

 \begin{figure}[!htb]
      \centerline{  \resizebox{8cm}{!}{%
      \includegraphics{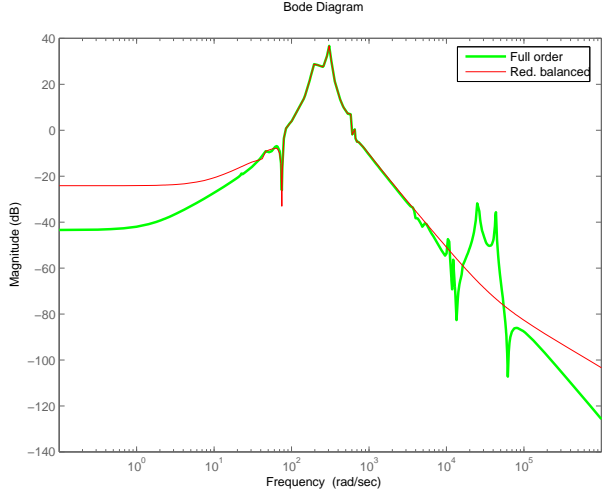}}}
      \caption{Model reduction} 
      \label{bal}
\end{figure}

\section{Useful formulas for the user}

In this section we give simplified formulas for calculating the 
transfer function  and state space representations
of the $W$ parameterized by $\xi$ and $\zeta$.
It is also shown how to determine the 
state-covariances and state-Markov parameters
from these representations of  $W$.
(Note that the problem considered in this paper is the 
inverse of determining the interpolation parameters 
from the model $W$.)
These formulas will be important for applying the method
proposed here.

\subsection{For transfer functions}

Since $W = \sigma G / \xi G$ is a quotient between two functions 
in $\cK$, it follows from (\ref{ss2tf}) that it can be written as 
a quotient of two polynomials
\begin{equation}
\label{Wtf}
W(z) =  \frac{b(z)}{a(z)}
\end{equation}
where 
\begin{equation}
b(z) = \det(I-z(A-B\sigma^*))-\det(I-zA)
\end{equation}
and 
\begin{equation}\label{a}
a(z) = \det(I-z(A-B\xi^*))-\det(I-zA).
\end{equation}
Clearly, $z=0$ is a zero of both $a$ and $b$ so it is cancelled
out, which leaves a $W$ of degree at most $n-1$.

To check if the resulting transfer function $W$ satisfies
the interpolation conditions it is convenient to 
use (\ref{Geval2}) to obtain

\[ \ip{G}{W} = \bar W(A) B = \bar b(A) \bar a(A)^{-1} B.\]

To determine the state-covariance $\cS$ corresponding to a 
particular $W$ we can use the following formula 
from \cite{TGselHA}

\begin{equation}
\cS = \Lambda \ip{GW}{GW} =\frac{1}{2}\Lambda(\Psi \cG + \cG \Psi^\star)
\end{equation}
where $\Psi= \bar f(A)$
and $f$ is the positive real part of $WW^\star$.
The function $f$ satisfies
\[ f+f^\star = WW^\star, \]
and it is clear that 
$f=d/a$, where $a$ is given by (\ref{a})
and 
$d$
solves the 
equation
\begin{equation}
b(z)b^\star(z) = d(z)a^\star(z) + a(z) d^\star(z).
\end{equation}
This equation has a unique solution $d$
such that all the roots of $d$ are outside the unit disc 
provided that all the roots of $a$ also are outside 
the unit disc \cite{Wilson,Goodman}.

\subsection{For state-space realizations}

A state-space realization of degree $n$ corresponding
to (\ref{Wtf}) is given by
\begin{equation}\label{ssdegn}
W(z) = \left[ z\sigma\beta \left( I- A\beta z \right)^{-1} AB +\sigma B
\right](\xi B)^{-1},
\end{equation}
where $\beta=I-B\xi(\xi B)^{-1}$.
From the last section it is known that a state-space  
realization of $W$ of degree $n-1$ exists, and one
is determined in the appendix.

Given a state space representation of $\cW$ as in (\ref{eq:ss}) 
the product $G\cW$ has a realization
\begin{equation}
\left( \begin{array}{c|c}
\hat A & \hat B \\ \hline
\hat C & \hat D
\end{array}
\right)
=
\left( \begin{array}{cc|c}
\cA & 0 & \cB \\
B\cC& A & B\cD \\
\hline
0 & I & 0
\end{array}
\right).
\end{equation}
Then the state-Markov parameter is given by
\begin{equation}
\ip{G}{\cW} = B\cD^*+ A \tilde P \cC^*, 
\end{equation}
where $\tilde P$ solves the 
Stein equation
$\tilde P = A \tilde P\cA^* +  B \cB^*$,
and the state-covariance $\cS$ is given by 
\begin{equation}
\cS =\Lambda \hat C \hat P \hat C^*, 
\end{equation}
where $\hat P$ solves the Lyapunov equation
$\hat P = \hat A\hat P\hat A^* + \hat B \hat B^*$.

\section{Conclusions and future work}

The ideas and results in \cite{PerIEEE-AC} were shown
to carry over to the case where not all interpolation points 
are at zero.
This freedom of chosing the interpolation points can be used
to obtain an improved matching at some frequency regions.
One example was given to illustrate the effect of 
moving the interpolation points.
Input-to-state filters proved to be a convenient tool to 
derive this theory and simple formulas based on solving
Lyapunov equations were obtained.
However, if really high order models are considered
specialized numerical tools have to be developed.

The approach used in \cite{Vanna} applies 
only the interpolation on $\Sigma$ and 
instead of matching $H$ arbitrary 
spectral zeros may be chosen.
This gives the user more freedom 
in designing the model, but at the 
price of having to tune more parameters.

In \cite{PerMTNS2006}, a generalization of the 
Markov and covariance interpolation problem 
with variable input variance to MIMO systems 
was considered.
A similar generalization should be possible here.


\section{Acknowledgement}

The author wishes to thank the anonymous referee 
providing valuable comments.

\appendix

We first derive (\ref{ssdegn}) 
\begin{eqnarray*}
 W(z) & = &\frac{\sigma(I-zA)^{-1}B}{\xi(I-zA)^{-1}B} \\
 &= & \frac{\sigma B+\sigma (z^{-1}I-A)^{-1}AB}{\xi B+\xi(z^{-1}I-A)^{-1}AB}\\
& = & \left( \sigma B+\sigma (z^{-1}I-A)^{-1}AB \right) \times \\
& & \left( I - \xi(z^{-1}I-A\beta)^{-1}AB (\xi B)^{-1} \right)
(\xi  B)^{-1}\\
& = & 
\left( \begin{array}{cc|c}
A\beta & 0 & AB(\xi B)^{-1} \\
-AB\xi & A & AB \\
\hline
-\sigma B\xi & \sigma & \sigma B
\end{array}\right) (\xi B)^{-1}.\\
& = & 
\left( \begin{array}{cc|c}
A\beta & 0 & AB(\xi B)^{-1} \\
0 & A & 0 \\
\hline
\sigma B\xi \beta & \sigma & \sigma B
\end{array}\right) (\xi B)^{-1}\\
& = & 
\left( \begin{array}{c|c}
A\beta  & AB \\
\hline
\sigma \beta & \sigma B
\end{array}\right) (\xi B)^{-1}.
\end{eqnarray*} 
Here we have used the matrix 
\[ T= \begin{block}{cc}
I & 0 \\ -(\xi B)I & I
\end{block}\]
to do a change of basis in the large
system before cancelling the 
unreachable second part of the state vector.

This realization is still non-minimal
since there is both a zero and a pole at infinity.
Note that 
\begin{equation} \label{betaB}
\beta B = B - (\xi B)^{-1} B \xi B = 0.
\end{equation}

Now we use $\Gamma$ in (\ref{Gamma}) to do a 
change of basis.
From (\ref{betaB}) it follows that 
\[ \sigma \beta \Gamma  = 
\sigma \beta A 
\begin{block}{cccc}
0 &B & \cdots& A^{n-2}B 
\end{block},\] 
\[ \Gamma^{-1}A \beta \Gamma  = 
\Gamma^{-1} A\beta A 
\begin{block}{cccc}
0 &B & \cdots& A^{n-2}B 
\end{block},\]
and 
\[ \Gamma^{-1} AB  = 
\begin{block}{ccccc}
0 &1  & 0& \cdots& 0 
\end{block}^T.\]
 
Then the first state in the new basis is 
not observable or reachable so a reduced order
realization is obtained by cancelling it:

\begin{equation}\label{Wreal}
 W(z) = 
\left( \begin{array}{c|c}
\hat \Gamma A\beta A\tilde \Gamma  & e_1 \\
\hline \rule{0cm}{.42cm}
\sigma \beta A \tilde \Gamma & \sigma B
\end{array}\right) (\xi B)^{-1},
\end{equation}
where 
\[ \tilde \Gamma \triangleeq
\begin{block}{cccc}
B & AB & \cdots& A^{n-2}B 
\end{block},\] 
and
\[ \hat\Gamma \triangleeq \left( \Gamma^{-1}\right)_{2:n} = 
\begin{block}{cc} 0 & I \end{block} \Gamma^{-1}, \]
and the subindex $2:n$ denotes rows $2$ to $n$ 
of the matrix.

In particular, if the characteristic polynomial $\chi_A(t)$ is parameterized as
\[ \chi_A(t) = t^n+\chi_1 t^{n-1}+ \cdots +\chi_{n-1} t + \chi_n, \]
the dynamics matrix in (\ref{Wreal}) is 
\begin{eqnarray*}
\tilde A & \triangleeq &
\begin{block}{cc} 0 & I \end{block} \Gamma^{-1}A \beta A \tilde \Gamma \\
 &= & \begin{block}{cc} 0 & I \end{block} 
\Gamma^{-1}A \left(\tilde \Gamma -(\xi B)^{-1}AB\xi \tilde \Gamma\right) \\
& =&  \begin{block}{ccccc} 0 &  -\chi_{n-1}\\ I & -\chi \end{block} -
(\xi B)^{-1} e_1 \xi \tilde \Gamma  \\
& =&  \begin{block}{ccccc} -\gamma &  -\gamma_{n-2}-\chi_{n-1}\\ 
I & -\chi \end{block} 
\end{eqnarray*}
where 
\[ \chi \triangleeq \begin{block}{cccc} \chi_{n-2} & \cdots & \chi_1 \end{block}^T,\]
\[ \gamma \triangleeq \begin{block}{cccc} 
\gamma_0 & \cdots & \gamma_{n-3} \end{block} 
\]
and $\gamma_k=\xi A^{k} B/ (\xi B)$ for $k=0,1,\cdots, n-2$.


\end{document}